\documentclass{amsart}

\usepackage{amssymb}

\newtheorem{theorem}{Theorem}[section]
\newtheorem{lemma}[theorem]{Lemma}

\newtheorem{proposition}[theorem]{Proposition}

\newtheorem{remark}[theorem]{Remark}

\numberwithin{equation}{section}

\newcommand{\Tr}{\text{Tr\,}}

\begin{document}
\setcounter{page}{1}

\thanks{Supported by the G\"oran Gustafsson Foundation (KVA), the ESF-network MISGAM and
the EU network ENIGMA}


\title[A multi-dimensional Markov chain and the Meixner ensemble]
{A multi-dimensional Markov chain and the Meixner ensemble}
\author[K.~Johansson]{Kurt Johansson}

\address{
Department of Mathematics,
Royal Institute of Technology,
SE-100 44 Stockholm, Sweden}

\email{kurtj@kth.se}

\begin{abstract}
We show that the transition probability of the Markoc chain \newline
$(G(j,1),\dots,G(j,n))_{j\ge 1}$,
where the $G(i,j)'s$ are certain directed last-passage times, is given by a determinant
of a special form. An analogous formula has recently been obtained by Warren in a Brownian
motion model. Furthermore we demonstrate that this formula leads to the Meixner ensemble
when we compute the distribution function for $G(m,n)$. We also obtain the Fredholm
determinant  representation of this distribution, where the kernel has a double contour integral
representation.

\end{abstract}

\maketitle

\section{Introduction}\label{sect1}

The starting point for the present paper are some nice results from the interesting paper
\cite{Wa} by J. Warren. Let $B_k(t)$, $1\le k\le n$, be independent Brownian motions started
at the origin and define $X_k(t)$, $k=1,\dots, n$, recursively by
\begin{equation}\label{1.1}
X_k(t)=\sup_{0\le s\le t}(X_{k-1}(s)+B_k(t)-B_k(s)),
\end{equation}
$t\ge 0$. The multi-dimensional Markov process $\mathbf{X}(t)=(X_1(t),\dots,X_n(t))$ at a fixed 
time is closely related to the largest eigenvalues of succesive principal submatrices of a 
GUE matrix. In fact, let $H=(h_{ij})_{1\le i,j\le n}$, be an $n\times n$ GUE matrix, i.e. 
distributed according to the probability measure $Z_n^{-1}\exp(-\Tr H^2)dH$ on the space of
$n\times n$ Hermitian matrices, and let $H_k=(h_{ij})_{1\le i,j\le k}$, $1\le k\le n$, be 
the principal submatrices. Then, if $\lambda_{\max} (M)$ denotes the largest eigenvalue of 
the Hermitian matrix $M$, we have $X(1/2)=(\lambda_{\max} (H_1),\dots,\lambda_{\max} (H_n))$
in distribution, \cite{Ba}, \cite{GTW}, \cite{Wa}.
Furthermore, there is a nice formula for the transition function of the Markov process
$\mathbf{X}(t)$, \cite{Wa},
\begin{equation}\label{1.2}
\mathbb{P}[\mathbf{X}(t)=y\,|\,\mathbf{X}(s)=x]=
\det(D^{j-i}\phi_{t-s}(y_j-x_i))_{1\le i,j\le n},
\end{equation}
if $x_1\le\dots\le x_n$, $y_1\le\dots\le y_n$.
Here $D$ denotes ordinary differentiation, $D^{-1}$ is anti-derivation,
\begin{equation}\label{1.3}
D^{-k}f(x)=\int_{-\infty}^x\frac{(x-y)^{k-1}}{(k-1)!}f(y)dy,
\end{equation}
and $\phi_t(x)=(2\pi t)^{-1/2}\exp(-x^2/2t)$ is the transition density for Brownian motion.

Let $F_{\text{GUE}(n)}(\eta)$ be the distribution function for the largest eigenvalue of an
$n\times n$ Hermitian matrix $H$ from GUE. It follows easily from (\ref{1.2}) that
\begin{equation}\label{1.4}
F_{\text{GUE}(n)}(\eta)=\det(D^{j-i+1}\phi_{1/2}(\eta))_{1\le i,j\le n}.
\end{equation}
This formula, given in \cite{Wa}, can also be obtained directly from the GUE eigenvalue measure,
see proposition \ref{prop2.3} below.

We will show in this paper, starting from definitions, that we have analogous formulas for
the vector $\mathbf{G}(i)=(G(i,1),\dots, G(i,n))$, $i\ge 0$ of certain last-passage times defined
as follows, \cite{JSh}. Let $w(i,j)$, $(i,j)\in\mathbb{Z}_+^2$, be independent geometric random
variables with parameter $q$, $0<q<1$,
\begin{equation}\label{1.5}
\mathbb{P}[w(i,j)=k]=(1-q)q^k
\end{equation}
and define
\begin{equation}\label{1.6}
G(m,n)=\max_{\pi}\sum_{(i,j)\in\pi} w(i,j),
\end{equation}
where the maximum is over all up/right paths from $(1,1)$ to $(m,n)$. It is clear that these 
random variables satisfy the recursion relation
\begin{equation}\label{1.6.2}
G(m,n)=\max(G(m-1,n),G(m,n-1))+w(m,n),
\end{equation}
where $G(0,n)=G(n,0)=0$ for $n\ge 1$. If we use this recursion relation repeatedly we see that
\begin{equation}
G(m,n)=\max_{1\le j\le n}(G(j,n-1)+\sum_{i=1}^mw(i,n)-\sum_{i=1}^{j-1} w(i,n)),
\notag
\end{equation}
which looks like a discrete version of (\ref{1.1}). From this it is reasonable to expect 
that there should be a formula for the transition function for the Markov chain
$(\mathbf{G}(i))_{i\ge 0}$ similar to (\ref{1.2}). This is indeed the case in a very natural
way where the differentiation operator is replaced by a finite difference operator, see
theorem \ref{thm2.1}. This will imply a formula similar to (\ref{1.3}) for the distribution
function of $G(m,n)$, see theorem \ref{thm2.2}. We will also show how we can go from this
formula (\ref{2.3}) to the known expressions \cite{Ok}, \cite{JDPNG}, for this 
distribution function in terms of the Meixner ensemble, \cite{JSh}, and as a Fredholm
determinant with a double contour integral expression for the kernel, \cite{Ok}, \cite{JDPNG}.
The Fredholm determinant formula has the advantage that it is much better suited for
computation of asymptotics.
The argument in this paper leading to (\ref{2.8}) gives an alternative approach, starting 
from the definitions, to this formula.

Results related to the transition probability (\ref{2.2}) for $(\mathbf{G}(i))_{i\ge 0}$,
but in the case of $w(i,j)$ exponentially distributed, go back to the work of G. Sch\"utz,
\cite{Sc}, where the totally asymmetric exclusion process (TASEP) is studied using the
Behte ansatz, see e.g. \cite{JSh} for a discussion of the relation to $G(m,n)$.
This is not exactly the same Markov chain, but the results of Sch\"utz can also be used
to derive the expression for the distribution function for $G(m,n)$ in terms of the Laguerre
ensemble, see \cite{RaSc} and also \cite{NaSa}. In \cite{RaSc} the case of geometric random 
variables is also considered and the formula for the distribution function for $G(m,n)$
in terms of the Meixner ensemble derived. This is based on results from \cite{BPS} on a
discrete TASEP-type model. More general formulas for the asymmetric exclusion process
(ASEP) have been proved recently in \cite{TrWi}. Results generalizing the formula (\ref{1.2})
to discrete models has also been given independently in \cite{DiWa}.

\section{Results}\label{sect2}


For $z\in\mathbb{Z}$ we define $w(x)=(1-q)q^xH(x)$, $0<q<1$, where $H$ is the Heaviside 
function,  $H(x)=0$ if $x<0$ and $H(x)=1$ if $x\ge 0$. The $m$-fold convolution of $w$
with itself is then the negative binomial distribution,
\begin{equation}\label{2.1}
w_m(x)=(1-q)^m\binom{x+m-1}{x}q^xH(x),
\end{equation}
as is not difficult to see using generating functions. 

For a function $f:\mathbb{Z}\to\mathbb{C}$ we denote by $\Delta$ the usual finite difference 
operator, $\Delta f(x)=f(x+1)-f(x)$. We also set
\begin{equation}
(\Delta^{-1}f)(x)=\sum_{y=-\infty}^{x-1} f(y)
\notag
\end{equation}
provided the sum is convergent, and $\Delta^0f=f$. Note that $\Delta(\Delta^{-1}f)=
\Delta^{-1}(\Delta f)=f$.

Let $W_n=\{x\in\mathbb{Z}^n\,;\, x_1\le \dots\le x_n\}$, and let $\mathbf{G}(i)$, $i\ge 0$, 
be the Markov chain defined in the introduction.

\begin{theorem}\label{thm2.1}
If $x,y\in W_n$, then for $m\ge \ell$,
\begin{equation}\label{2.2}
\mathbb{P}[\mathbf{G}(m)=y\,|\,\mathbf{G}(\ell)=x]=\det(\Delta^{j-i}
w_{m-\ell}(y_j-x_i))_{1\le i,j\le n}.
\end{equation}
\end{theorem}
We postpone the proof to section \ref{sect3}. The proof first shows (\ref{2.2}) in the case
$m=\ell+1$ using (\ref{1.6.2}) and an induction argument in $n$, and then establishes a
convolution formula for the determinants involved using the generalized Cauchy-Binet identity
(\ref{2.5}).

Taking $\mathbf{G}(0)=0$ it is not difficult to show, see section \ref{sect3}, that we have 
the following consequence of (\ref{2.2}), which is analogous to (\ref{1.4}).
\begin{theorem}\label{thm2.2}
For any $\eta\in\mathbb{N}$, $m\ge n\ge 1$,
\begin{equation}\label{2.3}
\mathbb{P}[G(m,n)\le\eta]=\det(\Delta^{j-i-1}w_m(\eta+1))_{1\le i,j\le n}.
\end{equation}
\end{theorem}
As stated in the introduction it is possible to relate the expression in the right hand side of
(\ref{1.3}) directly to the expression for the distribution function coming from the GUE
eigenvalue measure. Let
\begin{equation}
\Delta_n(x)=\det(x_i^{j-1})_{1\le i,j\le n}
\notag
\end{equation}
be the Vandermonde determinant.

\begin{proposition}\label{prop2.3}
We have the following identity for any $\eta\in\mathbb{R}$, $n\ge 1$,
\begin{equation}\label{2.4}
\det(D^{j-i+1}\phi_{1/2}(\eta))_{1\le i,j\le n}
=\frac 1{Z_n}\int_{(-\infty,\eta]^n}\Delta_n(x)^2\prod_{j=1}^ne^{-x_j^2}d^nx,
\end{equation}
where $Z_n$ is the appropriate normalization constant.
\end{proposition}

\begin{proof}
Let $H_j$, $j\ge 0$, be the standard Hermite polynomials. Then $H_j(x)=2^jp_j(x)$, where
$p_j$ is a monic polynomial, and we have Rodrigues' formula
\begin{equation}
D^je^{-x^2}=(-1)^jH_j(x)e^{-x^2}.
\notag
\end{equation}
Hence,
\begin{align}
D^{j-i-1}\phi_{1/2}(\eta)&=D^{-i}(D^{j-1}\phi_{1/2})(\eta)=\int_{-\infty}^{\eta}
\frac{(\eta-x)^{i-1}}{(i-1)!} D^{j-1}\phi_{1/2}(x)dx
\notag\\
&=\frac{2^{j-1}(-1)^{i+j}}{(i-1)!\sqrt{\pi}}\int_{-\infty}^{\eta}
(x-\eta)^{i-1}p_j(x)e^{-x^2}dx.
\notag
\end{align}
Using row and column operations we obtain
\begin{align}
\det(D^{j-i-1}\phi_{1/2}(\eta))_{1\le i,j\le n}&=\prod_{j=0}^{n-1}\frac{2^j}{j!\sqrt{\pi}}
\det(\int_{-\infty}^{\eta} x^{i-1}x^{j-1}e^{-x^2}dx))_{1\le i,j\le n}
\notag\\
&=\frac 1{Z_n}\int_{(-\infty,\eta]^n}\Delta_n(x)^2\prod_{j=1}^ne^{-x_j^2}d^nx,
\notag
\end{align}
with $Z_n=2^{-n(n-1)/2}\pi^{n/2}\prod_{j=0}^{n-1}j!$. In the last equality we have used 
the generalized Cauchy-Binet identity,
\begin{equation}\label{2.5}
\det(\int_X\phi_i(x)\psi_j(x)d\mu(x))=\frac 1{n!}
\int_{X^n}\det(\phi_i(x_j))\det(\psi_i(x_j))
\prod_{j=1}^nd\mu(x_j)
\end{equation}
where all determinants are $n\times n$.
\end{proof}

We have a similar identity relating the right hand side of (\ref{2.3}) to the Meixner
ensemble. The proof is a little more involved and we postpone it to section \ref{sect3}.

\begin{proposition}\label{prop2.4}
For any $\eta\in\mathbb{N}$, $m\ge n\ge 1$,
\begin{equation}\label{2.6}
\det(\Delta^{j-i-1}w_m(\eta+1))_{1\le i,j\le n}=
\frac 1{Z_{m,n}}\sum_{0\le x_i\le\eta+n-1}
\Delta_n(x)^2\prod_{j=1}^n\binom{x_j+m-n}{x_j}q^{x_j}.
\end{equation}
\end{proposition}
For asymptotic analysis it is more useful to have a representation of the distribution 
function for $G(m,n)$ as a Fredholm determinant with an appropriate kernel. It is
possible to go to such a formula using Meixner polynomials and a standard random matrix
theory computation as was done in \cite{JSh}. There is also another formula for the kernel
as a double contour integral which can be obtianed from the Schur measure, see \cite{Ok},
\cite{JDPNG} or \cite{JHou}. It is actually possible to go directly to a Fredholm determinant
formula with a double contour integral formula for the kernel starting from the
expression in the right hand side of (\ref{2.3}).

Let $\gamma_r$ denote a circle centered at the origin with radius $r>0$. Let $1<r_2<r_1<1/q$ 
and define
\begin{equation}\label{2.7}
K_{m.n}(x,y)=\frac 1{(2\pi i)^2}\int_{\gamma_{r_2}}\frac{dz}z\int_{\gamma_{r_1}}\frac{dw}w
\frac{w}{w-z}\frac{z^{x+n}}{w^{y+n}}\frac{(1-qz)^m(1-w)^n}{(1-z)^n(1-qw)^n}
\end{equation}
for $x,y\in\mathbb{Z}$.
\begin{proposition}\label{prop2.5}
For any integer $\eta\ge 0$, $m\ge n\ge 1$,
\begin{equation}\label{2.8}
\mathbb{P}[G(m,n)\le \eta]=\det(I-K_{m,n})_{\ell^2(\{\eta+1,\eta+2,\dots\})}.
\end{equation}
\end{proposition}

The proof will be given in the next section.

\begin{remark}\label{remark2.6}
\rm It would be interesting to understand the joint distribution of $G(m_i,n_i)$,
$i=1,\dots,p$ in order to understand the fluctuations of the ``last-passage times
surface'', $\mathbb{Z}_+^2\ni (m,n)\to(G(m,n)$. If $(m_1,n_1),\dots, (m_p,n_p)$ form
a right/down path then the joint distribution of $G(m_1,n_1),\dots, G(m_p,n_p)$
can be expressed as a Fredholm determinant, see \cite{JHou}, \cite{BoOl}, and it is
possible to investigate the asymptotic fluctuations. However, the asymptotic correlation
between for example $G(m,m)$ and $G(n,n)$, $m<n$, is not known and their is no nice expression
for their joint distribution. Using (\ref{2.2}) we can write down an expression for their joint
distribution, which was one of the motivations for the present work. We have
\begin{align}\label{2.9}
&\mathbb{P}[G(m,m)\le\eta_1, G(n,n)\le \eta_2]
\notag\\
&=\sum_{x\in W_n,\,\,x_m\le\eta_1}
\sum_{y\in W_n,\,\,y_n\le\eta_2}
\det(\Delta^{j-i}w_m(x_j))_{1\le i,j\le n}
\det(\Delta^{j-i}w_{n-m}(y_j-x_i))_{1\le i,j\le n}.
\end{align}
However, we have not been able to rewrite this in a form useful for asymptotic computations.
\it
\end{remark}

\begin{remark}\label{remark2.7} 
\rm The case when the $w(i,j)$'s are exponential random variables can be treated in a completely
analogous way or by taking the appropriate limit of the formula above, $q=1-\alpha/L$, $G(m,n)\to
G(m,n)/L$ and $l\to\infty$.
\it
\end{remark}

\begin{remark}\label{remark2.8}
\rm Random permutations can be obtained as a limit of the above model, 
$q=\alpha/n^2$, $n\to\infty$,
\cite{JDOPE}. The random variable $G(n,n)$ then converges to $L(\alpha)$ the Poissonized 
version of the length $\ell_N$ of a longest increasing subsequence of a random permutation
from $S_N$. We can take this limit in the formulas (\ref{2.7}) and (\ref{2.8}) and this leads
to a formula for $\mathbb{P}[L(\alpha)\le\eta]$ as a Fredholm determinant involving the
discrete Bessel kernel, \cite{JDOPE}, \cite{BOO}.
Hence, we obtain a new proof of this result which does not inolve some form of 
the RSK-correspondence.\it
\end{remark}

\section{Proofs}\label{sect3}

\subsection{Proofs of theorems \ref{thm2.1} and \ref{thm2.2}}\label{subsect3.1}

To prove theorem 2.1 we first consider the case $m-\ell=1$. The transition function from 
$\mathbf{G}(\ell)$ to $\mathbf{G}(m)$ is, by (\ref{1.6.2}),
\begin{equation}\label{3.1}
\mathbb{P}[\mathbf{G}(\ell+1)=y\,|\,\mathbf{G}(\ell)=x]=
\prod_{j=1}^n w(y_k-\max(x_k,y_{k-1})),
\end{equation}
where we have set $y_0=0$ and $x,y\in W_n$. Note also that it is clear from (\ref{1.6.2})
that $(\mathbf{G}(i))_{i\ge 0}$ is a Markov chain.
The right hand side can be written as a determinant by the following lemma.

\begin{lemma}\label{lem3.1}
If $x,y\in W_n$, then
\begin{equation}\label{3.2}
\prod_{j=1}^n w(y_k-\max(x_k,y_{k-1}))=\det(\Delta^{j-i}w(y_j-x_i))_{1\le i,j\le n}.
\end{equation}
\end{lemma}

\begin{proof}
We use induction with respect to $n$. the claim is trivial for $n=1$. Assume that it is true
up to $n-1$. Expand the determinant in (\ref{3.2}) along the last row,
\begin{align}\label{3.3}
&\det(\Delta^{j-i}w(y_j-x_i))_{1\le i,j\le n}=
\notag\\
&\sum_{k=1}^{n-2}(-1)^{k+n}\Delta^{k-n}w(y_k-x_n)\det(\Delta^{j-i}w(y_j-x_i))_{1\le i,j\le n-1,
j\neq k}
\notag\\
&-\Delta^{-1}w(y_{n-1}-x_n)\det(\det(\Delta^{j-i}w(y_j-x_i))_{1\le i,j\le n-1,
j\neq n-1}
\notag\\
&+w(y_n-x_n)\det(\Delta^{j-i}w(y_j-x_i))_{1\le i,j\le n-1}.
\end{align}
We will first show that each term in the sum from $k=1$ to $n-2$ in the right hand side
of (\ref{3.3}) is zero. Let $\Delta_y$ denote the difference operator with respect to the
variable $y$. The fact that $\Delta(\Delta^{-1}w)=w$, then gives
\begin{equation}
\det(\Delta^{j-i}w(y_j-x_i))_{1\le i,j\le n-1,j\neq k}=
\Delta_{y_{k+1}}\dots\Delta_{y_n}
\det(\Delta^{j-i}w(\tilde{y}_j-x_i))_{1\le i,j\le n-1},
\notag
\end{equation}
where $\tilde{y}_j=y_j$ if $1\le j\le k-1$, $\tilde{y}_j=y_{j+1}$ if $k\le n-1$.
By the induction asumption this equals
\begin{equation}
\Delta_{y_{k+1}}\dots\Delta_{y_n}
\prod_{j=1}^{k-1}w(y_j-\max (x_j,y_{j-1}))w(y_{k+1}-\max(x_k,y_{k-1}))\prod_{j=k+1}^{n-1}
w(y_{j+1}-\max(x_j,y_j)).
\notag
\end{equation}
If $y_k<x_n$, then $\Delta^{k-n}w(y_k-x_n)=0$ since $\Delta^{-j}w(x)=0$ if $x<j$. 
In this case the $k\,$'th term in the sum in (\ref{3.3}) is zero. Assume that $y_k\ge x_n$. 
Then $y_n\ge \dots\ge y_k\ge x_n$ and we obtain
\begin{align}
&\det(\Delta^{j-i}w(y_j-x_i))_{1\le i,j\le n-1,j\neq k}
\notag\\
&=\prod_{j=1}^{k-1}w(y_j-\max(x_j,y_{j-1}))(1-q)q^{-\max(x_k,y_{k-1})}
\Delta_{y_{k+1}}\dots\Delta_{y_n}
\prod_{j=k+1}^{n-1}(1-q)q^{y_{j+1}-y_j}.
\notag
\end{align}
But,
\begin{equation}
\Delta_{y_{k+1}}\dots\Delta_{y_n}
\prod_{j=k+1}^{n-1}(1-q)q^{y_{j+1}-y_j}
=(1-q)^{n-k-1}\Delta_{y_{k+1}}\dots\Delta_{y_n} q^{y_n}=0
\notag
\end{equation}
since $k+1<n$. Hence, each term in the sum from $k=1$ to $n-2$ in (\ref{3.3}) is zero and we
obtain
\begin{align}\label{3.4}
&\det(\Delta^{j-i}w(y_j-x_i))_{1\le i,j\le n}
\notag\\
&=-\Delta^{-1}w(y_{n-1}-x_n)\det(\Delta^{j-i}w(y_j-x_i))_{1\le i,j\le n-1,
j\neq n-1}
\notag\\
&+w(y_n-x_n)\det(\Delta^{j-i}w(y_j-x_i))_{1\le i,j\le n-1}.
\end{align}
If $y_{n-1}<x_n$, then $\Delta^{-1}w(y_{n-1}-x_n)=0$ and the right hand side of (\ref{3.4})
is
\begin{equation}
w(y_n-x_n)\prod_{j=1}^{n-1}w(y_j-\max(x_j,y_{j-1})),
\notag
\end{equation}
by the induction assumption. Since $y_{n-1}<x_n$, $w(y_n-x_n)=w(y_n-\max(x_n,y_{n-1})$
and we get exactly the right hand side of (\ref{3.2}).

Assume now that $y_{n-1}\ge x_n$. Note that $\Delta^j w(x)=(q-1)\Delta^{j-1}w(x)$ if
$j\ge 1$ and $x\ge 0$. The determinant
\begin{equation}
\det(\Delta^{j-i}w(y_j-x_i))_{1\le i\le n-1, j\neq n-1}
\notag
\end{equation}
has $\Delta^{n-i}w(y_n-x_i)=(q-1)\Delta^{n-1-i}w(y_n-x_i)$, $1\le i<n$, in the last
column, since $y_n\ge y_{n-1}\ge x_n\ge\dots\ge x_1$.
By the induction assumption this determinant equals
\begin{equation}
(q-1)\prod_{j=1}^{n-2} w(y_j-\max(x_j,y_{j-1}))w(y_n-\max(x_{n-1}, y_{n-2})).
\notag
\end{equation}
Thus, the right hand side of (\ref{3.4}) equals
\begin{align}\label{3.5}
&-(1-q)(q^{y_{n-1}-x_n}-1)\prod_{j=1}^{n-2} w(y_j-\max(x_j,y_{j-1}))w(y_n-\max(x_{n-1},
y_{n-2}))
\notag\\
&+w(x_n-y_n)\prod_{j=1}^{n-1}w(y_j-\max(x_j,y_{j-1})),
\end{align}
where we have used the fact that
$(\Delta^{-1}w)(y_{n-1}-x_n)(q-1)=(1-q)(q^{y_{n-1}-x_n}-1)$
for $y_{n-1}-x_n\ge 0$. since $y_n\ge y_{n-1}\ge x_n\ge\dots\ge x_1$, the expression in
(\ref{3.5}) can be written
\begin{align}
&(1-q)^2\left[-(q^{y_n-x_n-1}-1)q^{y_n-\max(x_{n-1},y_{n-2})}+
q^{y_n-x_n}q^{y_{n-1}-\max(x_{n-1},y_{n-2})}\right]
\notag\\
&\times\prod_{j=1}^{n-2}w(y_j-\max(x_j,y_{j-1}))
\notag\\
&=(1-q)q^{y_n-y_{n-1}}\prod_{j=1}^{n-1}w(y_j-\max(x_j,y_{j-1}))
=\prod_{j=1}^{n}w(y_j-\max(x_j,y_{j-1})).
\notag
\end{align}
\end{proof}

Theorem \ref{thm2.1} follows from lemma \ref{lem3.1} and the following convolution type
formula for determinants of the form we have.

\begin{lemma}\label{lem3.2}
Assume that $f,g:\mathbb{Z}\to\mathbb{C}$ are such that $f(x)=g(x)=0$ if $x<M$ for some
$M\in\mathbb{Z}$. Then,
\begin{equation}\label{3.6}
\sum_{y\in W_n}\det(\Delta^{j-i}f(y_j-x_i))\det(\Delta^{j-i}g(z_j-y_i))
=\det((\Delta^{j-i}f\ast g)(z_j-x_i)),
\end{equation}
where all determinants are $n\times n$.
\end{lemma}

\begin{proof}
We will make use of the following summation by parts formula
\begin{equation}\label{3.7}
\sum_{y=a}^b\Delta u(y-x)v(z-y)=\sum_{y=a}^b u(y-x)\Delta v(z-y)
+u(b+1-x)v(z-b)-u(a-x)v(z+1-a).
\end{equation}
The first step is to show that
\begin{align}\label{3.8}
&\sum_{y\in W_n}\det(\Delta^{j-i}f(y_j-x_i))\det(\Delta^{j-i}g(z_j-y_i))
\notag\\
&=\sum_{y\in W_n}\det(\Delta^{1-i}f(y_j-x_i))\det(\Delta^{i-1}g(z_i-y_j))
\end{align}
by repeated summation by parts. The left hand side of (\ref{3.8}) can be written
\begin{align}
&\sum_{y\in W_n}\Delta_{y_n}\det(\Delta^{1-i}f(y_1-x_i)\dots \Delta^{n-1-i}f(y_{n-1}-x_i)
\,\,\,\Delta^{n-1-i}f(y_{n}-x_i))
\notag\\
&\times\det(\Delta^{i-1}g(z_i-y_1)\dots\Delta^{i-n}g(z_i-y_n)).
\notag
\end{align}
Here we have used the summation by parts formula (\ref{3.7}) to sum $y_n$ between $y_{n-1}$
and $\infty$. The terms coming from $u(b+1-x)v(z-b)-u(a-x)v(z+1-a)$ in (\ref{3.7}) with
$b\to\infty$ and $a=y_{n-1}$ give $0$ since $\Delta^{i-n}g(z_i-b)=0$ if $b$ is 
large enough and the other term gives rise to a determinant with two equal columns containing
$\Delta^{n-1-i}f(y_{n-1}-x_i)$. The result is
\begin{align}
&\sum_{y\in W_n}\det(\Delta^{1-i}f(y_1-x_i)\dots \Delta^{n-1-i}f(y_{n-1}-x_i)
\,\,\,\,\Delta^{n-1-i}f(y_{n}-x_i))
\notag\\
&\times\det(\Delta^{i-1}g(z_i-y_1)\dots\Delta^{i-n+1}g(z_i-y_{n-1})\,\,\,\,
\Delta^{i-n+1}g(z_i-y_{n})).
\notag
\end{align}
We can now repeat this procedure with $y_{n-1}, y_{n-2},\dots,y_2$, which gives
\begin{align}
&\sum_{y\in W_n}\det(\Delta^{1-i}f(y_1-x_i)\,\,\,\Delta^{1-i}f(y_2-x_i)
\,\,\,\Delta^{2-i}f(y_3-x_i)\dots
\,\,\,\Delta^{n-1-i}f(y_{n}-x_i))
\notag\\
&\times\det(\Delta^{i-1}g(z_i-y_1)\,\,\,\Delta^{i-1}g(z_i-y_2)
\,\,\,\Delta^{i-2}g(z_i-y_3)\dots
\,\,\,\Delta^{i-n+1}g(z_i-y_{n})).
\notag
\end{align}
Again we repeat the summation by parts procedure with $y_n,\dots,y_3$, 
then with $y_n,\dots,y_4$ and so on until we get the right hand side of (\ref{3.8}).

Next, we apply the generalized Cauchy-Binet identity (\ref{2.5}) to the 
right hand side of (\ref{3.8}). This gives
\begin{equation}
\det(\sum_{y\in\mathbb{Z}}\Delta^{1-i}f(y-x_i)\Delta^{j-1}g(z_j-y)).
\notag
\end{equation}
To prove the lemma it remains to show that
\begin{equation}
\sum_{y\in\mathbb{Z}}\Delta^{1-i}f(y-x)\Delta^{j-1}g(z-y)=
\Delta^{j-i}(f\ast g)(z-x).
\notag
\end{equation}
If we set $h(x)=H(x-1)$, then $\Delta^{-1}f(x)=h\ast f(x)$ and hence
\begin{align}
&\sum_{y\in\mathbb{Z}}\Delta^{1-i}f(y-x)\Delta^{j-1}g(z-y)=
\Delta_z^{j-1}(h^{\ast(i-1)}\ast f)\ast g(z-x)
\notag\\
&=\Delta_z^{j-1}(\Delta^{1-i}(f\ast g))(z-x)=\Delta^{j-i}(f\ast g)(z-x).
\notag
\end{align}
Here $h^{\ast j}$ denotes the $j$-fold convolution of $h$ with itself.
\end{proof}

To prove theorem \ref{thm2.2} we note taht by theorem \ref{thm2.1}
\begin{align}
\mathbb{P}[G(m,n)\le\eta]&=\sum_{x_1\le \dots\le x_n\le \eta}\det(\Delta^{j-i}w_m(x_j))
\notag\\
&=\sum_{x_1\le \dots\le x_{n-1}\le \eta}\,\,\,
\sum_{x_n=x_{n-1}}^{\eta}\det(\Delta^{j-i}w_m(x_j)).
\notag
\end{align}
Now,
\begin{align}
&\sum_{x_n=x_{n-1}}^{\eta}\det(\Delta^{j-i}w_m(x_j))
\notag\\
&=\det(\Delta^{1-i}w_m(x_1)\dots\Delta^{n-1-i}w_m(x_{n-1})\,\,\,
\Delta^{n-1-i}w_m(\eta+1)-\Delta^{n-1-i}w_m(x_{n-1}))
\notag\\
&=\det(\Delta^{1-i}w_m(x_1)\dots\Delta^{n-1-i}w_m(x_{n-1})\,\,\,
\Delta^{n-1-i}w_m(\eta+1)).
\notag
\end{align}
Repeated use of this argument proves theorem \ref{thm2.2}.

\subsection{Proofs of propositions \ref{prop2.4} and \ref{prop2.5}}\label{subsect3.2}

The generating function for $w(x)$ is
\begin{equation}
\sum_{x\in\mathbb{Z}} w(x)z^x=\frac{1-q}{1-qz}
\notag
\end{equation}
and hence, since $w_m$ is the $m$-fold convolution of $w$ with itself,
\begin{equation}
w_m(x)=\frac{(1-q)^m}{2\pi i}\int_{\gamma_r}\frac{dz}{(1-qz)^mz^{x+1}},
\notag
\end{equation}
where the radius $r$ of the circle $\gamma_r$, centered at the origin, satisfies $0<r<1/q$.

It follows that
\begin{equation}\label{3.9}
\Delta^kw_m(x)=\frac{(1-q)^m}{2\pi i}\int_{\gamma_r}\frac{(1-z)^k}{(1-qz)^mz^{x+k+1}}dz,
\end{equation}
for $k\ge 0$ (and also for $k<0$ if $r<1$).

As noted above, if $h(x)=H(x-1)$, then $\Delta^{-1}f(x)=h\ast f(x)$ and hence
$\Delta^{-k}f(x)=(h^{\ast k}\ast f)(x)$. Using e.g. generating functions it is not
difficult to see that
\begin{equation}\label{3.10}
h^{\ast k}(x)=\frac{(x-1)^{[k-1]}}{(k-1)!}H(x-k),
\end{equation}
where $y^{[k]}=y(y-1)\dots(y-k+1)$ is the factorial power. We can write
\begin{equation}
\Delta^{j-i-1}w_m(x)=\Delta^{-i}(\Delta^{j-1}w_m)(x)=\sum_{y\in\mathbb{Z}}h^{\ast i}(x-y)
\Delta^{j-1}w_m(y).
\notag
\end{equation}
Note that $h^{\ast i}(x-y)=0$ if $x-y<i$ by (\ref{3.10}) and hence $h^{\ast i}(x-y)=0$ if $y>x-1$ for
any $i\ge 1$. We obtain
\begin{equation}\label{3.11}
\Delta^{j-i-1}w_m(\eta+1)=\sum_{y=-\infty}^\eta \frac{(\eta-y)^{[i-1]}}{(i-1)!}\Delta^{j-1}w_m(y).
\end{equation}
Fix $L\ge n-1$. By (\ref{3.11}) and some row operations we find
\begin{align}\label{3.12}
\det(\Delta^{j-i-1}w_m(\eta+1))_{1\le i,j\le n}&=\det(\sum_{y=-\infty}^\eta \frac{(y+L)^{i-1}}{(i-1)!}
(-1)^{i-1}\Delta^{j-1}w_m(y))_{1\le i,j\le n}
\notag\\
&=\det(\sum_{y=-L}^\eta \frac{(y+L)^{i-1}}{(i-1)!})(-1)^{j-1}\Delta^{j-1}w_m(y))_{1\le i,j\le n},
\end{align}
since it follows from (\ref{3.9}) that $\Delta^{j-1}w_m(y)=0$ if $y\le -n$ for $1\le j\le n$.

If we choose $L=n-1$ it follows from theorem \ref{thm2.2}, (\ref{3.12}) and the generalized
Cauchy-Binet identity (\ref{2.5}) that
\begin{equation}\label{3.13}
\mathbb{P}[G(m,n)\le\eta]=\prod_{j=1}^{n-1}\frac 1{j!}
\sum_{y_1,\dots,y_n=0}^{\eta+n-1}\Delta_n(y)\det((-1)^{j-1}\Delta^{j-1}w_m(y_i-n+1))_{1\le i,j\le n}.
\end{equation}
If $0\le k<n$, $m\ge n$, then
\begin{equation}\label{3.14}
(-1)^k\Delta^kw_m(y)=s_k(y)\prod_{\ell=k+1}^{m-1}(y+\ell)q^yH(y+n-1),
\end{equation}
where $s_k$ is a polynomial of degree $k$. To see this we can use (\ref{3.9}). We see that the 
integral in (\ref{3.9}) is zero if $x\le -(k+1)$, and in particular if $x\le -n$ for all $k$,
$0\le k<n$. If we make the change of variables $z\to 1/z$ and assume that $x\ge 1-m$, we find
\begin{align}
\Delta^kw_m(y)&=\frac{(1-q)^m}{2\pi i}\int_{\gamma_{1/r}}\frac{(z-1)^k}{(z-q)^m}z^{y+m-1}dz
\notag\\
&=\sum_{r=0}^{y+m-1}\binom{y+m-1}{r}\frac{(1-q)^m}{2\pi i}\int_{\gamma_{1/r}}
(z-1)^k(z-q)^{r-m}q^{y+m-1-r}dz.
\notag
\end{align}
This is a polynomial of degree $m-1$ in $y$ times $q^y$, and this polynomial has zeros at 
$-(k+1),\dots,1-m$. Consequently (\ref{3.14}) follows.

Furthermore we have the following determinantal identity. Let $p_j$, $j=0,\dots,n-1$ be 
polynomials 
of degree $j$ and $A_2,\dots, A_{n-1}$ constants. Then there is a constant $B$ such that
\begin{equation}\label{3.15}
\det(p_{j-1}(x_i)\prod_{k=j+1}^n(x_i+A_k))_{1\le i,j\le n}=B\Delta_n(x).
\end{equation}
This is not hard to see. Choose $c_{rj}$ so that
\begin{equation}
\sum_{r=1}^nx^{r-1}c_{rj}=p_{j-1}(x)\prod_{k=j+1}^n(x+A_k).
\notag
\end{equation}
Then, the left hand side of (\ref{3.15}) is
\begin{equation}
\det(\sum_{r=1}^nx_i^{r-1}c_{rj})_{1\le i,j\le n}=\Delta_n(x)\det C,
\notag
\end{equation}
and we have proved (\ref{3.15}) with $B=\det C$. Actually, according to \cite{Kr} we have
\begin{equation}
B=\prod_{j=1}^n(-1)^{j-1}p_{j-1}(-A_j),
\notag
\end{equation}
but we will not need this result.

By (\ref{3.13}), (\ref{3.14}) and (\ref{3.15}) we obtain
\begin{align}
\mathbb{P}[G(m,n)\le\eta]&=C_{m,n}\sum_{y_1,\dots,y_n=0}^{\eta+n-1}\Delta_n(y)^2
\prod_{j=1}^n q^{y_j}\prod_{k=n}^{m-1}(y_j+k+1-n)
\notag\\
&=\frac 1{Z_{m,n}}\sum_{y_1,\dots,y_n=0}^{\eta+n-1}\Delta_n(y)^2
\prod_{j=1}^n \binom{y_j+m-n}{y_j}q^{y_j},
\notag
\end{align}
for some constants $C_{m,n}$, $Z_{m,n}$. If we let $\eta\to\infty$ we see that $Z_{m,n}$ must 
be exactly the normalization constant in the Meixner ensemble. This proves proposition
\ref{prop2.4}.

We now turn to the proof of proposition \ref{prop2.5}. Write $K=m-n+1$ and define, for
$0\le j<n$, $x\in\mathbb{Z}$,
\begin{equation}
a_j(x)=\frac{q-1}{2\pi i}\int_{\gamma_{r_2}} z^{x-1}\frac{(qz-1)^{j+K-1}}{(z-1)^{j+1}}dz
\notag
\end{equation}
and
\begin{equation}
b_j(x)=\frac{1}{2\pi i}\int_{\gamma_{r_1}} \frac{(w-1)^j}{w^x(qw-1)^{j+k}}dw,
\notag
\end{equation}
where $1<r_2<r_1<1/q$. If $x\ge 0$, $0\le j<n$, then
\begin{equation}\label{3.16}
a_j(x)=\frac{(q-1)^{j+K}}{j!} p_j(x-n),
\end{equation}
where $p_j$ is a monic polynomial of degree 1. This follows from the computation
\begin{align}
a_j(x+n)&=\frac{q-1}{2\pi i}\int_{\gamma_{r_2}} z^{x+n-1}\frac{(qz-1)^{j+K-1}}{(z-1)^{j+1}}dz
\notag\\
&=\frac{q-1}{2\pi i}\int_{\gamma_{r_2}} \sum_{r=0}^{x+n-1}\binom{x+n-1}{r}(z-1)^{r-j-1}(qz-1)
^{j+K-1}dz
\notag\\
&=\frac{q-1}{2\pi i}\int_{\gamma_{r_2}} \sum_{r=0}^j\binom{x+n-1}{r}(z-1)^{r-j-1}
(qz-1)^{j+K-1}dz.
\notag
\end{align}
We see that this is a polynomial of degree $j$ in $x$ with leading coefficient 
\newline 
$(q-1)^{j+K}/j!$.

It follows from theorem \ref{thm2.2} and (\ref{3.12}) with $L=-n$ that
\begin{equation}\label{3.17}
\mathbb{P}[G(m,n)\le\eta]=\det\left(\sum_{y=0}^{\eta+n}(-1)^j\Delta^jw_m(y-n)\right)_{0\le i,j<n},
\end{equation}
where, by (\ref{3.9}),
\begin{equation}\label{3.18}
(-1)^j\Delta^jw_m(y-n)=\frac{(1-q)^m}{2\pi i}\int_{\gamma_{r_1}}\frac{(z-1)^j}{(1-qz)^{n+K-1}
z^{y+j-n+1}}dz.
\end{equation}
Set
\begin{equation}
c_{j\ell}=\binom{n-\ell-1}{j-\ell}H(j-\ell).
\notag
\end{equation}
Then,
\begin{equation}\label{3.19}
\sum_{j=0}^{n-1}\frac{(-1)^j}{(q-1)^{j+K}}\Delta^jw_m(y-n)c_{j\ell}=b_\ell(y).
\end{equation}
To show this it is sufficient to show that
\begin{equation}\label{3.20}
\sum_{j=0}^{n-1}\frac{(1-q)^{m-j-K}(-1)^{j+K}(z-1)^j}{(1-qz)^{n+K-1}z^{y+j-n+1}}c_{j\ell}
=\frac{(z-1)^\ell}{(qz-1)^{\ell+K}z^y}
\end{equation}
by (\ref{3.18}) and the definition of $b_\ell$. The identity (\ref{3.20}) can be rewritten as
\begin{equation}\label{3.21}
\sum_{j=0}^{n-1}(1-\frac 1z)^j(1-q)^{m-j-K}(-1)^{j+K}c_{j\ell}=
(-1)^{K}(1-qz)^{n-\ell-1}(1-z)^\ell \frac 1{z^{n-1}}.
\end{equation}
Set $w=1-1/z$. Then (\ref{3.21}) is equivalent to
\begin{equation}
\sum_{j=0}^{n-1}w^j(1-q)^{n-j-1}(-1)^jc_{j\ell}=(1-q-w)^{n-\ell-1}w^\ell(-1)^\ell.
\notag
\end{equation}
By the binomial theorem
\begin{align}
(1-q-w)^{n-\ell-1}w^\ell(-1)^\ell&=\sum_{r=0}^{n-\ell-1}\binom{n-\ell-1}r (-w)^{r+\ell}
(1-q)^{n-\ell-1-r}
\notag\\
&=\sum_{j=\ell}^{n-1}\binom{n-\ell-1}{j-\ell}(-1)^j(1-q)^{n-j-1}w^j,
\notag
\end{align}
and we have proved (\ref{3.19}). Note that $\det(c_{j\ell})=1$ and hence by (\ref{3.16}),
(\ref{3.17}) and (\ref{3.19}),
\begin{align}\label{3.22}
\mathbb{P}[G(m,n)\le\eta]&=\det\left(\sum_{y=0}^{\eta+n}\frac{(-1)^{i+K}}{i!}y^i
\frac{(-1)^{j}}{(q-1)^{j+K}}\Delta^jw_m(y-n)\right)_{0\le i,j<n}
\notag\\
&=\det\left(\sum_{y=0}^{\eta+n}a_i(y)\sum_{j=0}^{n-1}
\frac{(-1)^{j}}{(q-1)^{j+K}}\Delta^jw_m(y-n)c_{j\ell}\right)_{0\le i,\ell<n}
\notag\\
&=\det\left(\sum_{y=0}^{\eta+n}a_i(y)b_j(y)\right)_{0\le i,j<n}.
\end{align}

We will now use the fact that
\begin{equation}\label{3.23}
\sum_{y=0}^\infty a_j(y)b_k(y)=\delta_{jk}.
\end{equation}
To prove this we use the definitions of $a_j$ and $b_k$,
\begin{align}
\sum_{x=0}^\infty a_j(x)b_k(x)&=\frac{q-1}{(2\pi i)^2}\int_{\gamma_{r_2}}
\frac{dz}{z}\int_{\gamma_{r_1}}dw\left(\sum_{x=0}^\infty\left(\frac zw\right)^x\right)
\frac{(qz-1)^{j+K-1}(w-1)^k}{(z-1)^{j+1}(qw-1)^{k+K}}
\notag\\
&=\frac{q-1}{(2\pi i)^2}\int_{\gamma_{r_2}}
dz\int_{\gamma_{r_1}}dw\frac wz\frac 1{w-z}
\frac{(qz-1)^{j+K-1}(w-1)^k}{(z-1)^{j+1}(qw-1)^{k+K}}.
\notag
\end{align}
Since $1<r_2<r_1<1/q$ we see that $w=z$ is the only pole in the $w$-integral and 
hence by the residue theorem this equals
\begin{equation}
\frac{q-1}{2\pi i}\int_{\gamma_{r_2}}(qz-1)^{j-k-1}(z-1)^{k-j-1}dz.
\notag
\end{equation}
If $j<k$ the integral is zero by Cauchy's theorem. If $j>k$ we make the change of
variables $z\to 1/z$ and we see again that the integral is zero. When $j=k$ we get
\begin{equation}
\frac{q-1}{2\pi i}\int_{\gamma_{r_2}}(qz-1)^{-1}(z-1)^{-1}dz=1.
\notag
\end{equation}
Using (\ref{3.23}) in (\ref{3.22}) we get
\begin{align}
\mathbb{P}[G(m,n)\le\eta]&=\det(\delta_{ij}-\sum_{y=\eta+n+1}^\infty a_i(y)b_j(y))_{0\le i,j<n}
\notag\\
&=\det(I-K_{m,n})_{\ell^2(\{\eta+1,\eta+2,\dots\})},
\notag
\end{align}
where
\begin{align}
&K_{m,n}(x,y)=\sum_{j=0}^{n-1} a_j(x+n)b_j(y+n)
\notag\\
&=\frac{q-1}{2\pi i)^2}\int_{\gamma_{r_2}}\frac{dz}z\int_{\gamma_{r_1}}dw
\frac{z^{x+n}}{w^{y+n}}\frac{(qz-1)^{K-1}}{(qw-1)^K(z-1)}
\sum_{j=0}^{n-1}\left[\frac{(w-1)(qz-1)}{(z-1)(qw-1)}\right]^j
\notag\\
&=\frac{1}{2\pi i)^2}\int_{\gamma_{r_2}}\frac{dz}z\int_{\gamma_{r_1}}dw
\frac{z^{x+n}}{w^{y+n}}\frac{(1-qz)^{K-1}}{(1-qw)^{K-1}}
\left[\frac{(1-w)^n(1-qz)^n}{(1-z)^n(1-qw)^n}-1\right]\frac 1{w-z}
\notag\\
&=\frac{1}{2\pi i)^2}\int_{\gamma_{r_2}}\frac{dz}z\int_{\gamma_{r_1}}\frac{dw}w
\frac{z^{x+n}}{w^{y+n}}\frac{(1-qz)^{K-1}}{(1-qw)^{K-1}}
\frac{(1-w)^n(1-qz)^n}{(1-z)^n(1-qw)^n}\frac w{w-z}.
\notag
\end{align}
Here we have used the fact that by Cauchy's theorem the $-1$ term in the third line
gives a zero 
contribution in the $z$-integral since $r_1>r_2$. This proves proposition \ref{prop2.5}.

\begin{remark}
\rm The Meixner ensemble is related to Meixner polynomials since these polynomials are 
orthogonal on $\mathbb{N}$ with respect to the negative binomial weight. These
polynomials are not used explicitely in the computations above but figure in the background.
This can be seen from the Rodrigues' formula,
\begin{equation}
p_j^{(m,q)}(x)q^{j+x}\binom{x+m-1}{x}=\Delta^j\left[\binom{x+m-1}{x}q^x\prod_{k=0}^{j-1}
(x-k)\right],
\notag
\end{equation}
and the integral representation
\begin{equation}
p_j^{(m,q)}(x)=\frac{j!}{2\pi i}\int_{\gamma_r}\frac{(1-z/q)^x}{(1-z)^{x+K}}
\frac{dz}{z^{j+1}},
\notag
\end{equation}
with $0<r<1$, where $p_j^{(m,q)}(x)$ are the standard Meixner polynomials, \cite{KoSw}.
\end{remark}

\medskip
\noindent
{\bf Acknowledgement}: I thank Jon Warren for a discussion and for
sending me a preliminary version of \cite{DiWa}.

\end{document}